\documentclass[11pt,twoside,a4paper]{article}
\usepackage{amsmath,amssymb,amsthm}
\usepackage{hyperref}
\usepackage{graphicx,psfrag}

\newtheorem{thm}{Theorem}[section]
\newtheorem{lem}[thm]{Lemma}
\newtheorem{pro}[thm]{Proposition}

\theoremstyle{definition}
\newtheorem{exa}[thm]{Example}

\theoremstyle{remark}
\newtheorem{rem}[thm]{Remark}

\newcommand{\R}{\mathbb{R}}

\newcommand{\C}{\mathbb{C}}

\newcommand{\cM}{\mathcal{M}}

\newcommand{\cP}{\mathcal{P}}

\newcommand{\cR}{\mathcal{R}}

\newcommand{\al}{\alpha}
\newcommand{\be}{\beta}

\newcommand{\de}{\delta}

\newcommand{\om}{\omega}

\newcommand{\si}{\sigma}

\newcommand{\la}{\lambda}
\newcommand{\La}{\Lambda}
\renewcommand{\phi}{\varphi}

\newcommand{\id}{\operatorname{id}}

\newcommand{\aut}{\operatorname{Aut}}
\newcommand{\crr}{\operatorname{cr}}
\newcommand{\cd}{\operatorname{cd}}

\newcommand{\sing}{\operatorname{sing}}
\newcommand{\reg}{\operatorname{reg}}

\newcommand{\tr}{\operatorname{Tr}}
\newcommand{\sign}{\operatorname{sign}}

\newcommand{\set}[2]{\{#1:\,\text{#2}\}}
\newcommand{\sm}{\setminus}
\newcommand{\sub}{\subset}

\newcommand{\ov}{\overline}

\newcommand{\wh}{\widehat}

\begin{document}

\title{Symmetries of cross-ratios and the equation for M\"obius structures}

\author{Sergei Buyalo\footnote{Supported by RFFI Grant 20-01-00070}}

\date{}
\maketitle

\begin{abstract} We consider orthogonal representations 
$\eta_n:S_n\curvearrowright\R^N$
of the symmetry groups
$S_n$, $n\ge 4$, 
with
$N=n!/8$
motivated by symmetries of cross-ratios. For 
$n=5$
we find the decomposition of
$\eta_5$
into irreducible components and show that one of the components gives
the solution to the equations, which describe M\"obius structures
in the class of sub-M\"obius structures. In this sense, the condition defining M\"obius
structures is hidden already in symmetries of cross-ratios.
\end{abstract}

\section{Introduction}
\label{sect:indroduction}

This note is an extension of \cite{Bu16}. We consider orthogonal representations 
$\eta_n:S_n\curvearrowright\R^N$
of the symmetry groups
$S_n$, $n\ge 4$,
with
$N=n!/8$
motivated by symmetries of cross-ratios. For 
$n=5$
we find the decomposition of
$\eta_5$
into irreducible components and show that one of the components gives
the solution to the equations obtained in \cite{Bu16}, which describe M\"obius structures
in the class of sub-M\"obius structures. In this sense, the condition defining M\"obius
structures is hidden already in symmetries of cross-ratios.

\subsection{M\"obius and sub-M\"obius structures}
\label{subsect:moeb_submoebius}

Given a set 
$X$,
we consider ordered 4-tuples
$Q\in X^4$.
Such a
$Q$
is said to be {\em admissible} if it contains at most two equal items. 
Let
$\cP_4=\cP_4(X)$
be the set of admissible 4-tuples. A 4-tuple
$Q\in X^4$
is said to be {\em nondegenerate} or {\em regular} if all its items are pairwise different.
We denote the set nondegenerate 4-tuples as
$\reg\cP_4=\reg\cP_4(X)$.

A M\"obius structure
$M$
on a set 
$X$
is a class of equivalent semi-metrics on
$X$,
where two semi-metrics in
$M$
are equivalent if and only if for any nondegenerate 4-tuple
$Q\in X^4$
cross-ratios of the semi-metrics coincide.

It is convenient to use an alternative description of a M\"obius structure
$M$
by three cross-ratios
$$Q\mapsto \crr_1(Q)=\frac{|x_1x_3||x_2x_4|}{|x_1x_4||x_2x_3|};
  \crr_2(Q)=\frac{|x_1x_4||x_2x_3|}{|x_1x_2||x_3x_4|};
  \crr_3(Q)=\frac{|x_1x_2||x_3x_4|}{|x_2x_4||x_1x_3|}$$
for 
$Q=(x_1,x_2,x_3,x_4)\in\cP_4$,
whose product equals 1, where
$|x_ix_j|=d(x_i,x_j)$
for any semi-metric
$d\in M$.
There is a reason to replace cross-ratios
$\crr_k$, $k=1,2,3$,
by their logarithms. In that way we define a M\"obius structure
$M$
on a set 
$X$
as the map 
$M:\cP_4\to\ov L_4$,
$$M(Q)=(\ln\crr_1(Q),\ln\crr_2(Q),\ln\crr_3(Q)),$$
where
$L_4\sub\R^3$
is the 2-plane given by the equation
$a+b+c=0$,
and
$\ov L_4$
is an extension of
$L_4$
by 3 infinitely remote points
$A=(0,\infty,-\infty)$, $B=(-\infty,0,\infty)$, $C=(\infty,-\infty,0)$.
The infinitely remote points are added to take into account degenerate
4-tuples
$Q\in\cP_4$.

An important feature of our approach is presence of the symmetry groups
$S_n$
of 
$n$
elements for 
$n\ge 3$.
The group
$S_4$
acts on
$\cP_4$
by entries permutations of any
$Q\in\cP_4$.
The group
$S_3$
acts on
$\ov L_4$
by signed permutations of coordinates, where a permutation
$\si:\ov L_4\to\ov L_4$
has the sign 
``$-1$'' 
if and only if
$\si$
is odd. The map 
$M$
is equivariant with respect to the signed cross-ratio homomorphism,
\begin{equation}\label{eq:signed_cross-ratio_homomorphism}
M(\pi(Q))=\sign(\pi)\phi(\pi)M(Q) 
\end{equation}
for every
$Q\in\cP_4$, $\pi\in S_4$,
where
$\phi:S_4\to S_3$
is the cross-ratio homomorphism, see sect.~\ref{subsect:cross-ratio_homomorphism}.

A {\em sub-M\"obius structure} on
$X$
is a map 
$M:\cP_4\to\ov L_4$
with the basic property (\ref{eq:signed_cross-ratio_homomorphism}) and
a natural condition for values on degenerate admissible 4-tuples,
see sect.~\ref{subsect:submoeb_structures}. This notion is introduced
in \cite{Bu16}, where its importance is demonstrated by showing that on
the boundary at infinity 
$X$
of every Gromov hyperbolic space
$Y$
there is a canonical sub-M\"obius structure
$M$
invariant under isometries of
$Y$
such that the 
$M$-topology
on
$X$
coincides with the standard one. 

The main result of \cite{Bu16} is a criterion under which a sub-M\"obius
structure on
$X$
is a M\"obius structure, that is, generated by semi-metrics. This criterion
is expressed as linear equations (A), (B), see Theorem~\ref{thm:submoeb_moeb},
for the codifferential
$\de M:\cP_5\to\ov L_5$
defined on admissible 5-tuples, see sect.~\ref{subsect:def_codiff}. 
In other words, to check whether
of a sub-M\"obius structure is a M\"obius one, one needs to consider its natural 
extension to admissible 5-tuples. The linearity of the equations (A), (B) 
is the main reason why we take logarithms in the definition of M\"obius structures.

\subsection{$S_5$-symmetry and its irreducible components}

The group
$S_5$
acts by entries permutations of the space
$\cP_5\sub X^5$
of admissible 5-tuples. We define a natural representation
$\eta_5:S_5\curvearrowright V^5$
of
$S_5$
on a 15-dimensional space
$V^5=\R^5\otimes V^4$, 
where
$V^4:=\R^3\supset L_4$,
see sect.~\ref{subsect:canonical_action_S^n}, in a way that 
$\eta_5$
describe the symmetry of any sub-M\"obius structure 
$M$
regarded on admissible 5-tuples, that is, the codifferential
$\de M$
is 
$\eta_5$-equivariant,
$$\de M(\pi P)=\eta_5(\pi)\de M(P)$$
for every
$P\in\cP_5$
and every
$\pi\in S_5$.
In sect.~\ref{subsect:decomposition_eta_5_irreducible} we describe
the decomposition of
$\eta_5$
into irreducible components,
$$\psi=\chi^{32}+\chi^{2^21}+\chi^{21^3}+\chi^{1^5},$$
where
$\psi:S_5\to\C$
is the character of
$\eta_5$
and
$\chi^{\{\}}$
are prime characters of
$S_5$.

On the other hand, in sect.~\ref{sect:characteristic_functions}
we find an
$\eta_5$-invariant
version of the equations (A), (B), that is,
$\eta_5$-invariant
subspace
$\wh R\sub L_5$
of solutions to (A), (B). It turns out that
$\wh R$
is exactly
$\chi^{32}$-subspace
$R$
of the 
$\eta_5$-decomposition
into irreducible components, 
$\wh R=R$,
see Proposition~\ref{pro:symset_irreducible}. Since
$\dim R=5$,
we find in particular that
$\dim\wh R=5$.
This leads to the main result of this note

\begin{thm}\label{thm:moeb_irreducible} A sub-M\"obius structure
$M$
on a set 
$X$
is M\"obius if and only if
$\de M(\reg\cP_5)\sub R$,
where
$R\sub L_5$
is the irreducible component of the canonical
representation
$\eta_5$
that corresponds to the prime character
$\chi^{32}$.
\end{thm}

\section{Sub-M\"obius structures}
\label{sect:submoebius_structures}

\subsection{Admissible $n$-tuples and the cross-ratio homomorphism}
\label{subsect:cross-ratio_homomorphism}

Here we briefly recall what is a sub-M\"obius structure
$M$
on a set 
$X$
and what are conditions under which 
$M$
is a M\"obius structure.

Given an ordered tuple
$P=(x_1,\dots,x_k)\in X^k$
we use notation
$$P_i=(x_1,\dots,x_{i-1},x_{i+1},\dots,x_k)$$
for
$i=1,\dots,k$.
We define the set 
$\cP_n$
of the {\em admissible} 
$n$-tuples 
in
$X$, $n\ge 5$,
recurrently by asking for every
$P\in\cP_n$
the 
$(n-1)$-tuples
$P_i$, $i=1,\dots,n$,
to be admissible. The set 
$\sing\cP_n$
of 
$n$-tuples
$P\in\cP_n$
having two equal entries is called the {\em singular}
subset of
$\cP_n$.
Its complement
$\reg\cP_n=\cP_n\sm\sing\cP_n$
is called the {\em regular} subset of
$\cP_n$. 
The set
$\reg\cP_n$
consists of nondegenerate 
$n$-tuples.

The symmetry group
$S_n$
acts on
$\cP_n$
by permutations of the entries of every
$P\in\cP_n$.
We represent a permutation
$\pi\in S_n$
as
$\pi=i_1\dots i_n$
where
$i_k=\pi^{-1}(k)$, $k=1,\dots,n$.

The {\em cross-ratio} homomorphism
$\phi:S_4\to S_3$
can be described as follows: a permutation of a tetrahedron
ordered vertices 
$(1,2,3,4)$
gives rise to a permutation of opposite pairs of edges
$((12)(34),(13)(24),(14)(23))$
taken in this order. Thus the kernel 
$K$
of
$\phi$
consists of four elements 1234, 2143, 4321, 3412,
and is isomorphic to the dihedral group
$D_4$
of a square automorphisms. The group
$D_4$
is also called the {\em Klein group}.

We denote by 
$\sign:S_4\to\{\pm 1\}$
the homomorphism that associates to every odd permutation the sign 
``$-1$''.

\subsection{Sub-M\"obius structures}
\label{subsect:submoeb_structures}

We denote by
$L_4$
the subspace in
$\R^3$
given by
$a+b+c=0$
(subindex 
$4$
is related to 4-tuples rather than to the dimension of
$\R^3$).
We extend 
$L_4$
by adding to it thee points
$A=(0,\infty,-\infty)$, $B=(-\infty,0,\infty)$, $C=(\infty,-\infty,0)$,
$\ov L_4=L_4\cup\{A,B,C\}$.

A {\em sub-M\"obius structure} on a set
$X$
is given by a map
$M:\cP_4\to\ov L_4$,
which satisfies the following conditions
\begin{itemize}
 \item [(a)] $M$
is equivariant with respect to the signed cross-ratio homomorphism
$\phi$,
i.e.
$$M(\pi P)=\sign(\pi)\phi(\pi)M(P)$$
for every
$P\in\cP_4$
and every
$\pi\in S_4$.
 \item [(b)] $M(P)\in L_4$
if and only if
$P$
is nondegenerate, $P\in\reg\cP_4$;
 \item [(c)] $M(P)=(0,\infty,-\infty)$
for
$P=(x_1,x_1,x_3,x_4)\in\cP_4$.
\end{itemize}

\begin{rem}\label{rem:cone_sub_moeb} Note that if
$M$, $M'$
are sub-M\"obius structures on
$X$,
then
$M+M'$
and
$\al M$, $\al>0$,
are also sub-M\"obius structures on
$X$.
However,
$-M$
is not a sub-M\"obius structure in the sense of our definition, because
the condition (c) in this case is violated. Therefore, the set 
$\cM$
of sub-M\"obius structures on 
$X$
is a cone.
\end{rem}

A function
$d:X^2\to\wh\R=\R\cup\{\infty\}$
is called a {\em semi-metric}, if it is symmetric,
$d(x,y)=d(y,x)$
for each
$x$, $y\in X$,
positive outside of the diagonal, vanishes on the diagonal
and there is at most one infinitely remote point
$\om\in X$
for
$d$,
i.e. such that
$d(x,\om)=\infty$
for some
$x\in X\sm\{\om\}$.
In this case, we require that 
$d(x,\om)=\infty$
for all
$x\in X\sm\{\om\}$.
A metric is a semi-metric that satisfies the triangle inequality.

With every semi-metric
$d$
on
$X$
we associate the {\em M\"obius} structure
$M_d:\cP_4\to\ov L_4$
given by
$M_d(x_1,x_2,x_3,x_4)=(a,b,c)$,
where 
\begin{itemize}
 \item[] $a=\cd(x_1,x_2,x_3,x_4)=(x_1|x_4)+(x_2|x_3)-(x_1|x_3)-(x_2|x_4)$
 \item[] $b=\cd(x_1,x_3,x_4,x_2)=(x_1|x_2)+(x_3|x_4)-(x_1|x_4)-(x_2|x_3)$
 \item[] $c=\cd(x_2,x_3,x_1,x_4)=(x_2|x_4)+(x_1|x_3)-(x_1|x_2)-(x_3|x_4)$,
\end{itemize}
$(x_i|x_j)=-\ln d(x_i,x_j)$.
One easily checks that for every semi-metric
$d$
on
$X$
the M\"obius structure
$M_d$
associated with 
$d$
is a sub-M\"obius structure.

A triple
$A=(\al,\be,\om)\in X^3$
of pairwise distinct points is called a {\em scale triple}.
We use notation
$M(P_i)=(a(P_i),b(P_i),c(P_i))$, $i=1,\dots,5$, $P\in\cP_5$,
for a sub-M\"obius structure
$M$
on
$X$.
The following theorem has been proved in \cite[Theorem~3.4]{Bu16}.

\begin{thm}\label{thm:submoeb_moeb} A sub-M\"obius structure
$M$
on
$X$
is a M\"obius structure if and only if for every scale triple
$A=(\al,\be,\om)\in X^3$
and every admissible 5-tuple
$P=(x,y,A)\sub X$
the following conditions (A), (B) are satisfied
\begin{itemize}
 \item [(A)] $b(P_1)+b(P_4)=b(P_3)-a(P_1)$
for all
$y\in X$, $y\neq\al,\be$;
 \item[(B)] $b(P_2)=-a(P_4)+b(P_1)$,
where
$y\neq\al,\om$.
\end{itemize}
\end{thm}

\begin{rem}\label{rem:incerti} An alternative approach to Theorem~\ref{thm:submoeb_moeb}
is discussed in \cite{IM17}. 
\end{rem}

\section{$V^n$-spaces}
\label{sect:vn_spaces}

The tensor product
$\R^{m\otimes n}=\R^m\otimes\R^n$
taken over
$\R$
is the space of
$m\times n$
real matrices. We use notation
$V^4:=\R^3$
and inductively define
$V^{n+1}=\R^{n+1}\otimes V^n$
for 
$n\ge 4$.
Then 
$V^n$
is a real vector space of dimension
$\dim V^n=\frac{n!}{8}$
whose elements could be thought as matrices with
$n$
rows each of which being an element of
$V^{n-1}$,
in particular, we have a canonical basis
$e^n=e_i\otimes e^{n-1}$
of
$V^n$,
where
$e_i$, $i=1,\dots,n$,
is the canonical basic of
$\R^n$.
We also denote by
$W^n$, $n\ge 5$,
the vector space with
$V^n=W^n\otimes V^4$.

Recall that
$L_4\sub\R^3$
is the 2-plane given by the equation
$a+b+c=0$.
We put
$L_n=W^n\otimes L_4\sub V^n$, $\ov L_n=W^n\otimes\ov L_4$,
$n\ge 5$.
Furthermore,
$\dim L_n=2\dim W^n=\frac{2}{3}\dim V^n=\frac{n!}{12}$.

\subsection{The canonical action of $S_n$ on $V^n$}
\label{subsect:canonical_action_S^n}

Given
$\pi\in S_k$
and
$i\in\{1,\dots,k\}$
we let
$\pi_i\in S_{k-1}$
be the permutation of
$\{1,\dots,\wh i,\dots,k\}$
{\em induced} by
$\pi$,
that is, the restriction of the bijection
$\pi:\{1,\dots,k\}\to\{1,\dots,k\}$
to
$\{1,\dots,\wh i,\dots,k\}$,
$\pi_i=\pi|\{1,\dots,\wh i,\dots,k\}\to\{1,\dots,\wh{\pi(i)},\dots,k\}$,
where
$\wh i$
means that the 
$i$th
box is empty. For example, if
$\pi$
is a cyclic permutation,
$\pi(j)=j+1\mod k$,
then
$\pi_i$
is cyclic for every
$1\le i\le k-1$
while
$\pi_k=\id$.

\begin{lem}\label{lem:induced_permutations} For any
$\pi'$, $\pi\in S_k$
we have
$$(\pi'\pi)_i=\pi_{\pi(i)}'\circ\pi_i.$$
\end{lem}

\begin{proof} The right hand side of the equality is well defined
because
$\pi_i$
is a bijection from
$\{1,\dots,\wh i,\dots,k\}$
to
$\{1,\dots,\wh{\pi(i)},\dots,k\}$.
Furthermore, the both sides are well defined bijections from 
$\{1,\dots,\wh i,\dots,k\}$
to
$\{1,\dots,\wh{\pi'\pi(i)},\dots,k\}$.
Thus they coincide because
$\pi'(\pi(j))=(\pi'\pi)(j)$
for every
$j\in\{1,\dots,k\}$.
\end{proof}

Let
$e^n$
be the canonical basis of
$V^n$, $n\ge 4$.
Then
$e^{n+1}=e_i\otimes e^n$, 
where
$e_i$, $i=1,\dots,n+1$,
is the canonical basic of
$\R^{n+1}$.

Let
$\eta_4:S_4\to\aut(V^4)$
be the signed cross-ratio homomorphism,
$$\eta_4(\pi)(e_j)=\sign(\pi)e_{\phi(\pi)(j)},\quad j=1,2,3,$$
where
$\pi\in S_4$
and
$\phi:S_4\to S_3$
is the cross-ratio homomorphism. Now we define an action 
$\eta_{n+1}$
of
$S_{n+1}$
on
$V^{n+1}$
inductively by
$$\eta_{n+1}(\pi)(e^{n+1})=\set{e_{\pi(i)}\otimes\eta_n(\pi_i)e^n}{$i=1,\dots,n+1$}$$
for 
$\pi\in S_{n+1}$, $n\ge 4$.

\begin{exa}\label{exa:21345} Let
$\pi=21345\in S_5$.
Then 
$\pi_1=\pi_2=\id\in S_4$, $\pi_i=2134\in S_4$
for 
$i=3,4,5$.
For the cross-ratio homomorphism
$\phi:S_4\to S_3$
we have
$\phi(\pi_1)=\phi(\pi_2)=\id\in S_3$, $\phi(\pi_i)=132\in S_3$
for 
$i=3,4,5$.
Thus for 
$v\in V_5$,
$$v=\begin{pmatrix}
    a_1 &b_1 &c_1\\
    a_2 &b_2 &c_2\\
    a_3 &b_3 &c_3\\
    a_4 &b_4 &c_4\\
    a_5 &b_5 &c_5
    \end{pmatrix},$$
we have
$$\eta_5(\pi)(v)=\begin{pmatrix}
    a_2 &b_2 &c_2\\
    a_1 &b_1 &c_1\\
    -a_3 &-c_3 &-b_3\\
    -a_4 &-c_4 &-b_4\\
    -a_5 &-c_5 &-b_5
    \end{pmatrix}.$$
\end{exa}

\begin{lem}\label{lem:5-tuple_homomorphism} For every
$n\ge 4$
the map 
$\eta_n:S_n\to\aut(V^n)$
is a homomorphism, which is a monomorphism for 
$n\ge 5$.
\end{lem}

\begin{proof} It suffices to check that
$\eta_n(\pi'\pi)=\eta_n(\pi')\circ\eta_n(\pi)$
for any
$\pi$, $\pi'\in S_n$.
This is true for 
$n=4$
because
$\eta_4$
is a homomorphism by definition.
By induction we can assume that
$$\eta_n\left(\pi'\pi\right)e^n=\eta_n(\pi')\circ\eta_n(\pi)e^n$$
for 
$\pi$, $\pi'\in S_n$.
Using Lemma~\ref{lem:induced_permutations} we obtain
$$(\pi'\pi)_i=\pi_{\pi(i)}'\circ\pi_i$$
for 
$\pi$, $\pi'\in S_{n+1}$.
Then
$$\eta_{n+1}\left(\pi'\pi\right)e^{n+1}=
  e_{\pi'\pi(i)}\otimes\eta_n(\pi_{\pi(i)}')\eta_n(\pi_i)e^n=
  \eta_{n+1}(\pi')\left(e_{\pi(i)}\otimes\eta_n(\pi_i)e^n\right).$$
Hence
$\eta(\pi'\pi)=\eta(\pi')\circ\eta(\pi)$.
If
$\eta_n(\pi)=\id$
for
$n\ge 5$,
then
$\pi=\id$
by definition of
$\eta_n$.
\end{proof}

The action
$\eta_n$
of
$S_n$
on
$V^n$
is said to be {\em canonical}. Note that
$\eta_n$
preserves the space
$\ov L_n$.

\subsection{Decomposition of $\eta_5$ into irreducible components}
\label{subsect:decomposition_eta_5_irreducible}

In this section, we use elementary facts from the representation theory. The reader may consult
e.g. \cite{H}.

Let
$\rho$
be a linear representation of a group
$G$
in a vector space
$V$.
The function
$\chi_\rho:G\to\C$, $\chi_\rho(s)=\tr(\rho(s))$,
where
$\tr(\rho(s))$
is the trace of the linear map
$\rho(s):V\to V$,
is called the {\em character} of
$\rho$.
If
$\rho$
is irreducible, then the character
$\chi_\rho$
is {\em prime}.

The character
$\chi_\rho$
is constant on every class of conjugate elements of
$G$.
For the symmetry group
$S_n$
any class of conjugate elements is uniquely determined by
decomposition into cycles. We denote by
$1^a2^b3^c\dots$
the class in
$S_n$
having
$a$
cycles of length one,
$b$
cycles of length two,
$c$
cycles of length three etc. The number of classes is equal to the number of
subdivisions of
$n$
into integer positive summands. In its turn, the number of prime characters is equal to the number
of classes, and we denote the respective prime character of
$S_n$
by
$\chi^\la$,
where
$\la$ 
is the respective decomposition of 
$n$.
Here is the table of prime characters
of the group
$S_5$.
The first line is the list of the conjugate classes,
the second line lists the orders, i.e. the number of elements in a class.
The last line is the character of the canonical representation
$\eta_5$,
which is easily computed directly. For example, to compute
$\psi(21^3)$,
we take the transposition
$\pi=21345\in 21^3$
and note that there are precisely three elements
$e_3\otimes e_1$, $e_4\otimes e_1$, $e_5\otimes e_1$
of the canonical basis
$e^5$
of
$V^5$,
which are preserved up to the sign by
$\eta_5(\pi)$,
moreover,
$\eta_5(\pi)(v)=-v$
for all these elements
$v=e_3\otimes e_1,e_4\otimes e_1,e_5\otimes e_1$,
see Example~\ref{exa:21345}. Hence
$\psi(21^3)=-3$.
The list of the prime characters
can be found in \cite{H}.

\begin{center}
$\begin{array}{rrrrrrrr}
{\rm cycles} & 1^5 & 21^3 & 2^21 & 31^2 & 32 & 41 & 5\\
{\rm order} & 1 & 10 & 15 & 20 & 20 & 30 & 24\\
\chi^5 & 1 & 1 & 1 & 1 & 1 & 1 & 1\\
\chi^{41} & 4 & 2 & 0 & 1 & -1 & 0 & -1\\
\chi^{32} & 5 & 1 & 1 & -1 & 1 & -1 & 0\\
\chi^{31^2} & 6 & 0 & -2 & 0 & 0 & 0 & 1\\
\chi^{2^21} & 5 & -1 & 1 & -1 & -1 & 1 & 0\\
\chi^{21^3} & 4 & -2 & 0 & 1 & 1 & 0 & -1\\
\chi^{1^5} & 1 & -1 & 1 & 1 & -1 & -1 & 1\\
\psi   & 15 & -3 & 3 & 0 & 0 & -1 & 0
\end{array}$
\end{center}

A function
$f:G\to\C$
is {\em central}, if it is constant on conjugate classes.
The prime characters of a finite group
$G$
form an orthonormal basis of the central functions on 
$G$
with respect to the scalar product
$$\langle\psi,\chi\rangle=\frac{1}{g}\sum_{s\in G}\psi(s)\chi(s)^\ast,$$
where
$g=|G|$
is the order of
$G$, $\chi(s)^\ast$
is the complex conjugate of
$\chi(s)$.
We have

$$\langle\psi,\chi^\la\rangle= 
\begin{cases} 
1,\ \la=\{32\},\{2^21\},\{21^3\},\{1^5\}\\
0,\ \la=\{5\},\{41\},\{31^2\},
\end{cases}
$$
therefore the character
$\psi$
of
$\eta_5$
has the decomposition 
\begin{equation}\label{eq:decomposition_eta_5_irreducible}
\psi=\chi^{32}+\chi^{2^21}+\chi^{21^3}+\chi^{1^5} 
\end{equation}
into prime characters of
$S_5$.
The dimensions of prime components of the decomposition can be read off the 
first column of the list of prime characters, that is
$$15=5+5+4+1.$$
Thus the respective decomposition of
$\eta_5$
into irreducible representations can be described as follows.
The decomposition
$V^5=L_5\oplus L_5^\perp$
is
$\eta_5$-invariant.
The 5-dimensional subspace
$L_5^\perp\sub V^5$
orthogonal to
$L_5$
is decomposed into 1-dimensional and 4-dimensional invariant irreducible subspaces, 
which correspond to the characters 
$\chi^{1^5}$, $\chi^{21^3}$
respectively. The 1-dimensional 
$\chi^{1^5}$-subspace
is spanned by
$w=e_1-e_2+e_3-e_4+e_5$,
where
$e_i\in V^5$
has
$i$th 
row 
$(1,1,1)$
and zero the remaining rows. The 10-dimensional
$L_5$
is decomposed as
$$L_5=R\oplus R^\perp$$
into 5-dimensional irreducible subspaces, where
$R$
is a
$\chi^{32}$-subspace,
that is, the character of
$\eta_5|R$
is
$\chi^{32}$,
and
$R^\perp$
is a
$\chi^{2^21}$-subspace.

\section{Characteristic functions}
\label{sect:characteristic_functions}

\subsection{An algorithm}

Let
$Q=(a,b,c,d)$
be an ordered nondegenerate 4-tuple. The set of the opposite
pairs in
$Q$
has a natural order
$\{(ab,cd),(ac,bd),(ad,bc)\}$.
We regard an ordered nondegenerate 5-tuple
$P$
as the vertex set of a 4-simplex. Elements of every matrix
$v\in V^5$
can be labeled by 3-faces of
$P$
as follows. To every 3-face 
$P_i=P\sm i$, $i\in P$,
we associate the 
$i$th 
row of
$v$,
whose items are labeled by pairs of opposite edges of
$P_i$.
The order of the items is induced by the order of
$P_i$
as above, which in its turn is induced by the order of
$P$.

Let
$\La=\La(P)$
be the edges set of
$P$, $|\La|=10$.
With each
$\la\in\La$
we associate a function
$r_\la:V^5\to\R$
called {\em characteristic} as follows. Given 
$v\in V^5$,
we consider the rows
$v_i$, $v_j$, $v_k$
of
$v$,
where
$P\sm\la=\{i,j,k\}$,
which are labeled by the 3-faces
$P_i$, $P_j$, $P_k$.
The function
$r_\la$
only depends on items of these three rows. To obtain
$r_\la(v)$
we take those items, whose label contains
$\la$,
and take the sum of the chosen items with signs
$(-1)^{p+q+1}$,
where
$p$
is the column,
$q$
is the row of the item in the matrix formed by the rows
$v_i$, $v_j$, $v_k$.

As an example, we compute 
$r_\la(v)$
for 
$\la=25$,
where
$$v=\begin{pmatrix}
    a_1 &b_1 &c_1\\
    a_2 &b_2 &c_2\\
    a_3 &b_3 &c_3\\
    a_4 &b_4 &c_4\\
    a_5 &b_5 &c_5
    \end{pmatrix}.$$
In this case 
$i=1$, $j=3$, $k=4$
and

\begin{align*}
 v_1=&(2345,2435,2534)\\
 v_3=&(1245,1425,1524)\\
 v_4=&(1235,1325,1523),
\end{align*}
where we assume that
$P=(1,2,3,4,5)$
and use the cross-ratio labeling. The items containing
$\la=25$
are 
$2534$, $1425$, $1325$.
Hence
$$r_{25}(v)=-c_1-b_3+b_4$$
Note that Eq.(A) of Theorem~\ref{thm:submoeb_moeb}
is equivalent to
$r_{25}(v)=0$
because
$a_1+b_1+c_1=0$.
Similarly, Eq.(B) of Theorem~\ref{thm:submoeb_moeb} is equivalent to
$r_{35}(v)=0$.

It this way, we obtain the following list of the characteristic functions

\begin{align*}
r_{25}(v)&=-c_1-b_3+b_4
&r_{35}(v)&=b_1-b_2-a_4\\
r_{15}(v)&=-c_2+c_3-c_4
&r_{45}(v)&=-a_1+a_2-a_3\\
r_{12}(v)&=-a_3+a_4-a_5
&r_{13}(v)&=-a_2-b_4+b_5\\
r_{14}(v)&=b_2-b_3-c_5
&r_{23}(v)&=-a_1+c_4-c_5\\
r_{24}(v)&=b_1+c_3+b_5
&r_{34}(v)&=-c_1+c_2-a_5. 
\end{align*}

Usually, we consider the functions
$r_\la$
on the subspace
$L_5\sub V^5$,
and extend them on
$\ov L_5$,
see sect.~\ref{sect:vn_spaces}, with values in
$\R\cup\{+\infty\}\cup\{-\infty\}$. In this case we use the agreement
$(a+\infty)-(b+\infty)=0$
for any
$a$, $b\in\R$.
 
\subsection{Action of $S_5$ on characteristic functions}

The action of
$S_n$
on functions
$f:V^n\to\R$
induced by
$\eta_n$
is defined by
$$(\pi f)(x)=f(\eta_n(\pi)^{-1}x)$$
for 
$\pi\in S_n$, $x\in V^n$.

\begin{pro}\label{pro:orbit_functions} The group
$S_5$
preserves the set 
$\cR=\set{\pm r_\la}{$\la\in\La$}$
and acts on it transitively.
\end{pro}

\begin{proof} The function
$r_\la$
is uniquely determined up to the sign by the choice
of  
$\la$.
Thus
$\pi r_\la=\pm r_{\pi \la}$
for every
$\la\in\La$, $\pi\in S_5$.
One easily checks that the transposition
$\pi_\la\in S_5$
that permutes vertices of
$\la$
changes the sign of
$r_\la$, $\pi_\la r_\la=-r_\la$.
Since
$S_5$
acts transitively of the edge set
$\La$, 
it acts transitively on
$\cR$.
\end{proof}

\begin{pro}\label{pro:submoeb_moeb_log} A sub-M\"obius structure
$M$
on
$X$
is a M\"obius structure if and only if for every nondegenerate 5-tuple
$P=(x,y,\al,\be,\om)$
conditions (A), (B) of Theorem~\ref{thm:submoeb_moeb} are satisfied.
\end{pro}

\begin{proof} We fix a scale triple
$A=(\al,\be,\om)\in X^3$
and consider an admissible 5-tuple
$P=(x,y,A)$.
By Theorem~\ref{thm:submoeb_moeb} we only need to show that if
$P$
is degenerate, then (A), (B) are satisfied automatically. 

Note that
$r_{25}(M(P))=0$
is equivalent to (A), and
$r_{35}(M(P))=0$
is equivalent to (B).
Thus using Proposition~\ref{pro:orbit_functions} it suffices to check that if
$x=y$,
then
$r_\la(M(P))=0$
for every
$\la\in\La$.
Note that in this case
$x=y\not\in A$. 

We have
$M(P_1)=M(P_2)=(a,b,c)\in L_4$,
while
$M(P_3)=M(P_4)=M(P_5)=(0,\infty,-\infty)$
by definition of a sub-M\"obius structure. One directly checks that
$r_\la(M(P))=0$
for 
$\la=35,45,12,34$
using that these equations do not involve
$\pm\infty$.

The equation
$r_{25}(v)=0$
is equivalent to
$b_1+b_4=-a_1+b_3$.
In the case
$v=M(P)$
the both sides are equal to
$\infty$,
thus
$r_{25}(M(P))=0$
according to our agreement. Similar argument shows that 
$r_\la(M(P))=0$
for the remaining
$\la=15,13,14,23,24$,
hence
$r_\la(M(P))=0$
for all
$\la\in\La$.
\end{proof}

The functions
$r_\la$, $\la\in\La$,
on
$V^5$
are the characteristic functions for M\"obius structures. A point
$a\in L_5\sub V^5$
is said to be {\em root} of
$r_\la$
if
$r_\la(a)=0$.
For roots in
$L_5$
we use notation
$$R_\la=\set{a\in L_5}{$r_\la(a)=0$},$$
and
$\wh R=\cap_\la R_\la\sub L_5$.
By Proposition~\ref{pro:orbit_functions}, the 
$\eta_5$-action 
of
$S_5$
on
$L_5$
permutes the sets
$R_\la$, $\la\in\La$,
and
$\wh R$
is
$\eta$-invariant.
We call
$\wh R$
the {\em symmetry} set.

\section{Codifferential of a sub-M\"obius structure}

\subsection{Definition of a codifferential}
\label{subsect:def_codiff}

Recall that
$\ov L_4=L_4\cup\{A,B,C\}$,
see sect.~\ref{subsect:submoeb_structures}. For 
$n\ge 5$
we put 
$\ov L_n=W^n\otimes\ov L_4$,
see sect.~\ref{sect:vn_spaces}.

Let
$M$
be a sub-M\"obius structure on a set 
$X$.
We define its {\em codifferential} as the map
$\de M:\cP_5\to\ov L_5$
given by
$$\de M(P)_i=M(P_i)\in\ov L_4,$$
where 
$\de M(P)_i$
is the
$i$th
row of
$\de M(P)$, $i=1,\dots,5$.
This terminology is related to the fact that
$\de M$
is computed on a 5-tuple 
$P\in\cP_5$
via values of
$M$
on the ``boundary''
$dP=P_1\cup\dots\cup P_5$.

\begin{lem}\label{lem:codiff_5_submoebius} Codifferential of any
sub-M\"obius structure
$M$
on a set 
$X$
has the following properties
\begin{itemize}
 \item [(i)] $\de M$
is equivariant with respect to the homomorphism
$\eta_5$,
i.e.
$$\de M(\pi P)=\eta_5(\pi)\de M(P)$$
for every
$P\in\cP_5$
and every
$\pi\in S_5$;
 \item [(ii)] $\de M(P)\in L_5$
if and only if
$P$
is nondegenerate, $P\in\reg\cP_5$;
 \item [(iii)] for every admissible 5-tuple
$P=(x,x,\al,\be,\om)$
the equality
$$\de M_i(P)=\begin{cases} 
 (a,b,c)\in L_4,&i=1,2\\
 (0,\infty,-\infty),&i=3,4,5
\end{cases}$$
holds.
\end{itemize}
\end{lem}

\begin{proof} By construction, properties of sub-M\"obius structures 
and Lemma~\ref{lem:5-tuple_homomorphism} the map 
$\de M$
is equivariant w.r.t. the natural action of
$S_5$
on
$\cP_5$
and the action
$\eta_5$
on
$\ov L_5$,
$$\de M(\pi P)=\eta_5(\pi)\de M(P)$$
for every
$P\in\cP_5$.
This is (i).

Property~(ii) follows from the similar property for 
sub-M\"obius structures.

For 
$P=(x,x,\al,\be,\om)\in\sing\cP_5$
we have
$P_1=P_2=(x,\al,\be,\om)\in\reg(\cP_4)$,
and we put 
$M(P_1)=M(P_2)=(a,b,c)\in L_4$.
For
$P_3=(x,x,\be,\om)$, $P_4=(x,x,\al,\om)$
and
$P_5=(x,x,\al,\be)$
we have
$M(P_3)=M(P_4)=M(P_5)=(0,\infty,-\infty)$.
Thus (iii) holds.
\end{proof}

By property (ii) of Lemma~\ref{lem:codiff_5_submoebius},
$\de M(\reg\cP_5)\sub L_5$. 
Thus we can consider the restriction
$\de M|\reg\cP_5$
as a map with values in
$L_5$.

\begin{pro}\label{pro:moebius_codifferential} The following
conditions for a sub-M\"obius
structure 
$M$
on a set 
$X$
are equivalent:

\begin{itemize}
 \item [(i)] $M$
is a M\"obius structure;
 \item [(ii)] $\de M(\reg\cP_5)\sub R_\la$
for some
$\la\in\La$;
 \item [(iii)] $\de M(\reg\cP_5)\sub\wh R$.
\end{itemize}
\end{pro}

\begin{proof} By Theorem~\ref{thm:submoeb_moeb} we have
$\de M(\reg\cP_5)\sub R_{25}\cap R_{35}$
for a M\"obius structure 
$M$. 
It follows from Proposition~\ref{pro:orbit_functions} that
$\de M(\reg\cP_5)\sub R_\la$
for every
$\la\in\La$,
in particular,
$(i)\Longrightarrow(ii)$.

$(ii)\Longrightarrow(iii)$: 
Assume 
$\de M(\reg\cP_5)\sub R_\la$
for a sub-M\"obius structure
$M$
and some
$\la\in\La$.
Using
$\eta_5$-equivariance
of
$\de M$,
see Lemma~\ref{lem:codiff_5_submoebius}(i),
we obtain that
$\de M(\reg\cP_5)\sub R_\la$
for every
$\la\in\La$,
thus
$\de M(\reg\cP_5)\sub\wh R$.

$(iii)\Longrightarrow(i)$:
If
$\de M(\reg\cP_5)\sub\wh R$,
then in particular
$\de M(\reg\cP_5)\sub R_{25}\cap R_{35}$.
Thus
$M$
is M\"obius by Proposition~\ref{pro:submoeb_moeb_log}.
\end{proof}

\subsection{Proof of the main theorem}
\label{subsect:proof_main}

\begin{pro}\label{pro:symset_irreducible} The symmetry set 
$\wh R\sub L_5$
coincides with the 
$\chi^{32}$-subspace $R$
of the decomposition
$L_5=R\oplus R^\perp$
of
$\eta_5$
into irreducible components,
$\wh R=R$.
\end{pro}

\begin{proof} The set 
$\wh R$
is the intersection of hyperplanes
$R_\la$, $\la\in\La$,
thus
$\wh R\neq L_5$.
By Proposition~\ref{pro:moebius_codifferential}, we have
$\de M(\reg\cP_5)\sub\wh R$
for any M\"obius structure
$M$
on any set 
$X$.
Taking e.g. the M\"obius structure 
$M$ 
of the extended real line
$\wh\R=\R\cup\{\infty\}$, 
we see that
$\de M(\reg\cP_5)\neq\{0\}$.
This shows that
$\wh R$
is a proper subspace of
$L_5$.
Since
$\wh R$
is
$\eta_5$-invariant,
we conclude that
$\wh R$
coincides with
$R$
or
$R^\perp$.
Note that the characters
$\chi^{32}$, $\chi^{2^21}$
differ on the class 
$21^3$,
i.e., on any transposition. To make a choice, we compute
$\tr(\eta_5(\pi))$
for 
$\pi=15342\in 21^3$
on the orthogonal complement 
$\wh R^\perp$
to
$\wh R$.
It is spanned by the normal vectors
$n_\la$, $\la\in\La$,
to the hyperplanes
$r_\la=0$
in
$L_5$. 

For the transposition
$\pi=15342$
we have 
$$\eta_5(\pi)(w)=\begin{pmatrix}
                 -b_1 &-a_1 &-c_1\\
                 b_5 &c_5 &a_5\\
                 -c_3 &-b_3 &-a_3\\
                 -c_4 &-b_4 &-a_4\\
                 c_2 &a_2 &b_2
                \end{pmatrix}\quad\text{for any}\quad
w=\begin{pmatrix}
                 a_1 &b_1 &c_1\\
                 a_2 &b_2 &c_2\\
                 a_3 &b_3 &c_3\\
                 a_4 &b_4 &c_4\\
                 a_5 &b_5 &c_5
                \end{pmatrix}.$$

The hyperplane
$r_{25}=0$
in
$V^5$
is given by the equation
$-c_1-b_3+b_4=0$
with the normal vector 
$$v_{25}=\begin{pmatrix}
    0 & 0 & -1\\
    0 & 0 & 0\\
    0 & -1 & 0\\
    0 & 1 & 0\\
    0 & 0 & 0
    \end{pmatrix}\in V^5.$$
We take as
$n_{25}$
the projection of
$3v_{25}$
to 
$L_5$,
$$n_{25}=\begin{pmatrix}
    1 & 1 & -2\\
    0 & 0 & 0\\
    1 & -2 & 1\\
    -1 & 2 & -1\\
    0 & 0 & 0
    \end{pmatrix}.$$
Similarly, we obtain all ten vectors
$n_\la\in\wh R^\perp$, $\la\in\La$.

Then
$\eta_5(\pi)(n_\la)=-n_\la$
for 
$\la=25,13,14,34$,
and there is no any other
$\la\in\La$
such that
$\eta_5(\pi)(n_\la)=\pm n_\la$.
The four vectors
$n_\la$, $\la=25,13,14,34$,
are linearly depended,
$n_{25}+n_{13}=n_{14}+n_{34}$,
and they span a 3-dimensional subspace in
$\wh R^\perp$
on which
$\eta_5(\pi)$
acts as
$-\id$.
Therefore,
$\tr(\eta_5(\pi))|\wh R^\perp$,
which must be
$\pm 1$,
cannot be equal to
$1$,
hence
$\tr(\eta_5(\pi))=-1=\chi^{2^21}(\pi)$.
This shows that
$\wh R^\perp=R^\perp$
and thus
$\wh R=R$.
\end{proof}

Theorem~\ref{thm:moeb_irreducible} now follows from Propositions~\ref{pro:moebius_codifferential}
and ~\ref{pro:symset_irreducible}.

\bigskip
\begin{tabbing}

Sergei Buyalo\\

St. Petersburg Dept. of Steklov Math. Institute RAS,\\ 
Fontanka 27, 191023 St. Petersburg, Russia\\
{\tt sbuyalo@pdmi.ras.ru}\\

\end{tabbing}

\end{document}